\documentclass{sig-alt-full}

\usepackage[utf8]{inputenc}
\usepackage[T1]{fontenc}
\usepackage{lmodern}

\usepackage{comment}
\usepackage{amssymb,amsmath}
\usepackage{theorem}
\usepackage{graphics}
\usepackage{amsfonts}
\usepackage{mathrsfs}
\usepackage{amscd}
\usepackage{color}
\usepackage{url}
\usepackage{alltt}
\usepackage{bbm}

\usepackage{microtype}  
\usepackage{mathtools}  
\usepackage{booktabs}   
\usepackage{array}      

\newdef{defi}{Definition}
\newtheorem{thm}{Theorem}
\newtheorem{prop}{Proposition}
\newtheorem{conj}{Conjecture}
\newtheorem{cor}{Corollary}

\begin{document}
\crdata{}

\title{Symbolic integration of hyperexponential 1-forms}
\newfont{\authfntsmall}{phvr at 11pt}
\newfont{\eaddfntsmall}{phvr at 9pt}

\numberofauthors{3} 
\author{
\alignauthor
Thierry Combot\\
       \affaddr{Univ. de Bourgogne (France)}\\
       \email{$\!\!\!\!\!\!\!\!\!\!\!\!$thierry.combot@u-bourgogne.fr$\!\!\!\!\!\!\!\!\!\!\!\!$}
}
\date{today}

\maketitle

\begin{abstract} 
Let $H$ be a hyperexponential function in $n$ variables $x=(x_1,\dots,x_n)$ with coefficients in a field $\mathbb{K}$, $[\mathbb{K}:\mathbb{Q}] <\infty$, and $\omega$ a rational differential $1$-form. Assume that $H\omega$ is closed and $H$ transcendental. We prove using Schanuel conjecture that there exist a univariate function $f$ and multivariate rational functions $F,R$ such that $\int H\omega= f(F(x))+H(x)R(x)$. We present an algorithm to compute this decomposition. This allows us to present an algorithm to construct a basis of the cohomology of differential $1$-forms with coefficients in $H\mathbb{K}[x,1/(SD)]$ for a given $H$, $D$ being the denominator of $dH/H$ and $S\in\mathbb{K}[x]$ a square free polynomial. As an application, we generalize a result of Singer on differential equations on the plane: whenever it admits a Liouvillian first integral $I$ but no Darbouxian first integral, our algorithm gives a rational variable change linearising the system.
\end{abstract}

\bigskip
\vspace{1mm}
\noindent
{\bf Categories and Subject Descriptors: 68W30} \\

\vspace{1mm}
\noindent {\bf Keywords: Hermite Reduction, Function Decomposition, Symbolic Integration}

\section{Introduction}

Let us note  $\mathbb{K}(x)_n$ the set of $1$-forms with coefficients in $\mathbb{K}(x)$, where $\mathbb{K}$ is a finite extension of $\mathbb{Q}$. A closed $1$-form $\omega\in \mathbb{K}(x)_n$ defines after integration, up to addition of a constant, a $n$ variables (possibly multivalued) function. As $\omega$ has rational coefficients, its integral can be written as a linear combination of rational and logs of rational functions. Thanks to that, the hyperexponential functions, whose logarithmic differential is a closed rational $1$-form, are not so mysterious, being the exponential of such functions.

Things become more difficult whenever we want to iterate the process. Liouvillian functions are built by successive integrations, exponentiations and algebraic extensions. So one of the natural next steps is to study closed $1$-forms with hyperexponential functions coefficients. Acting the Galois group on the hyperexponential functions, it is always possible to write the $1$-form as a sum of closed $1$-forms in $H\mathbb{K}(x)_n$ for several hyperexponential functions $H$. Thus this comes down to studying the symbolic integration of closed differential forms in $H\mathbb{K}(x)_n$. One of the motivations is the following result

\begin{prop}[Singer \cite{Singer}]
Let $V$ be a rational vector field in the plane. If $V$ admits a Liouvillian first integral, then it admits a first integral of the form
$$\int e^{\int F(x,y)dx+G(x,y)dy} (V_2(x,y)dx-V_1(x,y)dy)$$
where $F,G$ are rational functions.
\end{prop}

The inner exponential integral defines a hyperexponential function $H$, and then the first integral is defined as the integral of the closed $1$-form $H(V_2dx-V_1dy)$. However, this integral is typically not elementary. In fact, Singer proved that a necessary (but not sufficient) condition for elementary integration is $H$ to be algebraic. Instead of trying to integrate in elementary terms, we will try to write it using univariate functions composed with rational functions and hyperexponential functions. This can be seen as a generalization of elementary functions in which only $\exp,\ln$ are allowed as transcendental univariate functions. In other terms, is it possible to explicitly integrate hyperexponential $1$-forms by extending our tool kit with all single variable functions? Admitting the Schanuel conjecture, we will prove the following

\begin{thm}\label{thm1}
Let us consider $\omega\in\mathbb{K}(x)_n$ and a hyperexponential function $H$ with $dH/H\in\mathbb{K}(x)_n$. Assume $H\omega$ is closed and that $H$ is transcendental. Then
\begin{itemize}
\item There exist a rational function $F\in\overline{\mathbb{K}}(x)$ and a function $J(z,x)$ rational in $x$ such that
\begin{equation}\label{eq4}
\int H\omega =H(x)J(F(x),x)
\end{equation}
\item If $H\omega$ is not exact, there exist rational functions $f,g\in\overline{\mathbb{K}}(z),\;R,T\in \overline{\mathbb{K}}(x)$ such that
\begin{equation}\label{eq5}
H(x)=T(x)\exp{\int^{F(x)} g(z) dz}
\end{equation}\begin{equation}\label{eq5b}
\int H\omega =\int^{F(x)} f(z)e^{\int g(z) dz} dz+ H(x)R(x)
\end{equation}
\end{itemize}
\end{thm}

Remark that when $H\omega$ is exact, we can still write expression \eqref{eq5b} (but not always \eqref{eq5}), simply taking $f=0,g=0$, as the expression just becomes the condition of exactness of $H\omega$. Theorem \ref{thm1} is effective, and we will present in section $4.2$ an algorithm computing possible $F,T,R,f,g$. It appears moreover that it is always possible to choose them with coefficients in $\mathbb{K}$. The decomposition given by equations \eqref{eq5},\eqref{eq5b} is not unique. As we will see in section $4$, there is choice: $F$ can be chosen up to homographic transformation, the function $T$ can be multiplied by an arbitrary rational function in $F$. This allows to put $f=1$ by changing $g$. Integration by parts of $\int f(z)\exp{\int g(z) dz} dz$ can change the expression of $R$ by changing $f$.

It appears that when $H$ is transcendental, only one new single variable Liouvillian function, $\int f(z)\exp{\left(\int g(z) dz\right)} dz$, is necessary to express the integral. When $H$ is algebraic, the decomposition does not always exist, and as in the elementary case, several logs can be necessary for expressing the integral.

Given a square free polynomial $S\in\mathbb{K}[x]$ and a hyperexponential function $H$ and $D$ the denominator of $dH/H$, we can now consider the vector space $\mathcal{V}$ of closed $1$-forms in $H\mathbb{K}[x,1/(SD)]_n$. The cohomology of this vector space is its quotient by the subspace of exact forms. Theorem \ref{thm1} will allow us to compute a basis of such space. Indeed, the functions $F,T,g$ only depend on $H$, and $R$ is irrelevant as $\int H\omega$ is computed modulo exact forms. Thus the computation of representation \eqref{eq5b} of an element of $\mathcal{V}$ reduces to the determination of $f$ for which we will be able to control its poles thanks to $S$ and generalized Hermite reduction \cite{bostan2013hermite}.

\begin{thm}\label{thm2}
Given a transcendental hyperexponential function $H$, $D$ the denominator of $dH/H$ and a square free polynomial $S$, the cohomology of the vector space of closed $1$-forms in $H\mathbb{K}[x,1/(SD)]_n$ is finite dimensional and a basis is computed by algorithm \underline{\sf CohomologyBasis}.
\end{thm}

We do not give an estimation of the cost of these algorithms because they manifest dependence to exponents. In particular, $R$ can have an arbitrary large degree even for given $H$ and bounded degree of $\omega$. The degree of the function $F$ can be controlled by the degree of the $F_i$ in $H$, but not by the degree of $dH/H$, as the degree of $F$ depends on rational relations between the residues of $dH/H$. As an application, we build the following algorithm, linearising differential equation by a rational variable change whenever it has a Liouvillian first integral (and none of lesser type).

\begin{cor}\label{cor1}
Let us consider the differential equation in dimension $1$
\begin{equation}\label{eq3}
\frac{\partial y}{\partial x}=V(x,y),\quad V\in\mathbb{K}(x,y).
\end{equation}
If \eqref{eq3} admits a Liouvillian first integral which is not a $k$-Dar\!bou\-xian (see definition in \cite{cheze2017symbolic}) nor exponential of Darbouxian, then there exists a rational variable change $X,Y\in\mathbb{K}(x,y)$ and $a,b\in\mathbb{K}(z)$ such that
$$\frac{\partial Y}{\partial X}=a(X)+b(X)Y$$
\end{cor}

The Liouvillian first integrals of equation \eqref{eq3} up to degree $N$ (with a suitable notion of degree) can be found algorithmically, see \cite{cheze2017symbolic}. The condition on the first integral is equivalent to $H$ transcendental and $H\omega$ not exact, which are both done in Theorem \ref{thm1}.

\section{Hyperexponential functions}

\begin{defi}
A hyperexponential function $H$ in $n$ variables $x=(x_1,\dots,x_n)$ is a function satisfying $dH/H\in \mathbb{K}(x)_n$.
\end{defi}

We have that all logarithmic partial derivatives of $H$ are rational, i.e.
$$\partial_1H=F_1(x)H,\dots \partial_nH=F_n(x)H$$
with $F_i\in\mathbb{K}(x)$. This PDE system is holonomic of rank $1$, which implies that $H$ is a $D$-finite function, and that the $F_i$ (which are the coefficients of $dH/H$) define $H$ up to multiplication by a constant.

\begin{defi}
The trace of $\alpha\in \mathbb{L}$ where $\mathbb{L}$ is a field extension of $\mathbb{Q}$, noted $tr(\alpha)$, is minus the second leading coefficient of the monic minimal polynomial in $\mathbb{Q}$ of $\alpha$. We say that $\alpha$ is traceless if $tr(\alpha)=0$.
\end{defi}

\begin{prop}\label{prop1}
A hyperexponential function can be written under the form
$$e^{F_0(x)}A(x)^{1/q}\prod\limits_{i=1}^p F_i(x)^{\lambda_i}$$
with $q\in\mathbb{N}^*$, $\lambda_i\in\mathbb{L}$ traceless independent over $\mathbb{Q}$, $A,F_0\in\mathbb{K}(x),F_1,\dots,F_p\in \mathbb{L}(x)$ and $\mathbb{L}$ is a finite algebraic extension of $\mathbb{K}$.
The algorithm \underline{\sf RationalIntegration} computes from the rational one form $dH/H$ such a representation.
\end{prop}

\noindent\underline{\sf RationalIntegration}\\
\textsf{Input:} A rational closed differential form $\omega\in \mathbb{K}(x)_n$.\\
\textsf{Output:} Rational fractions $A,F_0\in \mathbb{K}(x),F_1,\dots, F_p\in\mathbb{L}(x)$, $q\in\mathbb{N}^*$, $\lambda_1,\dots,\lambda_p\in\mathbb{L}$ traceless independent over $\mathbb{Q}$ where $\mathbb{L}=\mathbb{K}(\lambda_1,\dots,\lambda_p)$, such that a solution of $dH/H=\omega$ is
$$e^{F_0(x)}A(x)^{1/q}\prod\limits_{i=1}^p  F_i(x)^{\lambda_i}.$$
\begin{enumerate}
\item Let $F_0=0,A=1$ and $Res=[\;\!]$. From $i$ from $n$ to $1$ do
\begin{enumerate}
\item Compute the Hermite reduction of $\omega_i-\partial_i F_0$ in $\mathbb{K}(x_1,\dots,x_{i-1})(x_i)$, giving
$$\int \omega_i dx_i=\int \frac{P}{Q} dx_i+R.$$
\item Factorize $Q=Q_1\dots Q_l$. For $j$ from 1 to $l$ do
\begin{enumerate}
\item Compute the resultant of $(Q_j,P-\lambda \partial_i Q_j)$ and select all factors depending in $\lambda$ only, giving a polynomial $S_j\in\mathbb{K}[\lambda]$.
\item Compute $\tilde{S}_j= \prod_{\sigma\in Gal(\mathbb{K}:\mathbb{Q})} \sigma(S_j)$, and note $-s_j$ the quotient of its second leading coefficient by its leading coefficient.
\item Redefine $A=A Q_j^{s_j/{\deg \tilde{S}_j}}$.
\item For all roots $\lambda$ of $S_j$, add
$$[\lambda-s_j/{\deg \tilde{S}_j},gcd(Q_j,P-\lambda \partial_i Q_j)]$$
to the list $Res$, and redefine $F_0=F_0+R$.
\item Redefine $\omega$ as the differential form
\begin{equation}\label{eq1}
\!\!\!\!\!\!\!\!\omega-dR-\sum_{\lambda\in S_j^{-1}(0)} \lambda \frac{d(gcd(Q_j,P-\lambda \partial_i Q_j))}{gcd(Q_j,P-\lambda \partial_i Q_j)}
\end{equation}
\end{enumerate}
\end{enumerate}
\item Compute the minimal $q\in\mathbb{N}^*$ such that $A^q\in\mathbb{K}(x)$ and redefine $A=A^q$.
\item Compute a basis $B$ of the $\mathbb{Z}$ module generated by the $(Res_{i,1})$, and a integer passage matrix $M$.
\item Return
$$\left[F_0,\left[A,\frac{1}{q}\right],\left[ B_j, \left(\prod\limits_{i=1}^{\sharp Res} Res_{i,2}^{M_{i,j}}\right)\right]_{j=1\dots \sharp B}\right].$$
\end{enumerate}

Using this algorithm, from a closed $1$ form $\omega$, we obtain the elementary expression of the hyperexponential function $H$ such that $dH/H=\omega$. Remark that the most costly step from a theoretical point of view is the computation of the field $\mathbb{L}$, which is given by splitting fields of polynomials $S$, thus of the size order of $(\deg S)!$. However, the explicit computation of the field is necessary to compute the basis $B$ to ensure the $\lambda_i$ are traceless independent over $\mathbb{Q}$. This condition is not necessary to write down $H$ under elementary form, but will be for the algorithm of Theorem \ref{thm1}.

\begin{proof}
We will prove the termination and correctness of the algorithm, which will also give us the existence of such elementary expression. We first recall the result of Trager algorithm \cite{trager1976algebraic}, which takes in input a rational fraction in $\mathbb{K}(z)$ and returns an integral of the form
$$F_0(z)+\sum_{i=1}^p \lambda_i \ln F_i(z)$$
with $\lambda_i \in \overline{\mathbb{K}}$, $F_0\in\mathbb{K}(z)$ and $F_i\in\mathbb{K}(\lambda_1,\dots,\lambda_p)(z)$.

Let us first prove by recurrence the following properties are true at the beginning of the $j$ loop
\begin{itemize}
\item The form $\omega$ is a closed rational $1$-form with coefficients in $\mathbb{K}(x_1,\dots,x_i)$
\item We have $F_0\in\mathbb{K}(x)$ and $A^q\in\mathbb{K}(x)$ for some $q\in\mathbb{N}^*$.
\item The last $n-i$ coefficients of $\omega$ are zero.
\item The first entries of $Res$ are traceless.
\item The quantity $\omega+dF_0+\frac{dA}{A}+\sum_{k=1}^{\sharp Res} Res_{k,1} \frac{d Res_{k,2}}{Res_{k,2}}$ is invariant
\end{itemize}

For $i=n,j=1$, this is true by hypothesis on the input. Now assume it is true for some $i,j$.
In step $(a)$, we compute the Hermite reduction, ensuring that $Q$ has only simple factors in $x_i$. In step $(i)$, the roots of the resultant are the residues associated to the roots of $Q_j$. A priori, the residues are elements $\overline{\mathbb{K}(x_1,\dots,x_{i-1})}$. However, we also know that $\int \omega_i dx_i$ should have all its other derivatives in $x_1,\dots,x_{i-1}$ rational. This requires that all the residues to be constant with respect to these derivations. As they are in $\overline{\mathbb{K}(x_1,\dots,x_{i-1})}$, they should then be in $\overline{\mathbb{K}}$. Thus all the roots in $\lambda$ of the resultant are constant, and thus are roots of $S_j$.

As $Q_j$ is irreducible, its roots are conjugated, and so are the residues. Thus $S_j$ is a power of an irreducible polynomial, and so it $\tilde{S}$. Now shifting the roots by $s_j/{\deg \tilde{S}_j}$ ensure that their minimal polynomial has now $0$ as second leading coefficient. Thus $\lambda-s_j/{\deg \tilde{S}_j}$ is traceless, and thus so are all the first entries of $Res$. 

Let us now remark that
\begin{equation}\begin{split}\label{eq9}
\sum_{\lambda\in S_j^{-1}(0)} \lambda \frac{d(gcd(Q_j,P-\lambda \partial_i Q_j))}{gcd(Q_j,P-\lambda \partial_i Q_j)}=\\
\frac{s_jdQ_j}{Q_j\deg \tilde{S}_j}+\sum_{\lambda\in S_j^{-1}(0)} (\lambda-s_j/{\deg \tilde{S}_j}) \frac{d(gcd(Q_j,P-\lambda \partial_i Q_j))}{gcd(Q_j,P-\lambda \partial_i Q_j)}
\end{split}\end{equation}
In steps $(iii),(iv)$, we add to the list $Res$ the terms forming the right-hand side of \eqref{eq9}, we multiply $A$ by $Q_j^{s_j/{\deg \tilde{S}_j}}$ thus forming the $\frac{s_jdQ_j}{Q_j\deg \tilde{S}_j}$ of the right-hand side of \eqref{eq9}. We still have $A^q \in \mathbb{K}(x)$ for some $q\in\mathbb{N}^*$. Now the quantity
$$\frac{dA}{A}+\sum_{j=1}^{\sharp Res} Res_{j,1} \frac{d Res_{j,2}}{Res_{j,2}}$$
increases by \eqref{eq9}. Also $F_0$ increases by $R$, and so stays in $\mathbb{K}(x)$. This is compensated in step $(v)$ where $\omega$ decreases by the same quantities. Thus
$$\omega+dF_0+\frac{dA}{A}+\sum_{k=1}^{\sharp Res} Res_{k,1} \frac{d Res_{k,2}}{Res_{k,2}}$$
stays invariant. The expression $R$ plus equation \eqref{eq9} is an integral of $\omega$ in $x_i$, thus after the redefinition of $\omega$, we have $\omega_i=0$, and so $\int \omega$ does not depend on $x_i$. As is depended only in $x_1,\dots,x_i$ by hypothesis, it now only depend on $x_1,\dots,x_{i-1}$. It is still closed as we have subtracted to it a closed form.

We now unroll the recurrence properties up to $i=0$, and apply the redefinition of $A$ in step $(2)$, giving that
\begin{equation}\label{eq2}
F_0+\frac{dA}{qA}+\sum_{i=1}^{\sharp Res} Res_{i,1}\ln Res_{i,2}
\end{equation}
is an integral of the original $\omega$, and thus its exponential is an expression of $H$ as $H=\exp \int \omega$. The new $A$ is now in $\mathbb{K}(x)$ as required.

Now in step $(3)$, we compute a basis of the $\mathbb{Z}$-module of the $(Res_{i,1})_{i=1\dots \sharp Res}$. Thus any $Res_{i,1}$ can be written as an integer linear combination of the $B$, which is given by the $i$-th line of passage matrix $M$. In step $(4)$, the integral \eqref{eq2} is rewritten using the additive rules of the logs. The fact that the entries of $M$ are integers ensures that the products are rational, and the elements of $B$ are $\mathbb{Q}$ independent by construction. They are traceless as the trace is a $\mathbb{Q}$-linear function. The field $\mathbb{L}$ is given by $\mathbb{K}(B)$.
\end{proof}

\section{Existence theorem}

In this section, we will prove Theorem \ref{thm1}. The necessary number theory result we will need for our proofs is the following

\begin{conj}\label{conj1}
Given $\lambda_1,\dots,\lambda_p\in\mathbb{L}$ traceless linearly independent over $\mathbb{Q}$, the numbers
$$e^{i\pi\lambda_1},\dots,e^{i\pi\lambda_p}$$
are algebraically independent.
\end{conj}

For $p=1$, this conjecture is true thanks to the result of Gelfond \cite{gelfond1934septieme}. For $p\geq 2$, it is implied by Schanuel conjecture. For our proof of Theorem \ref{thm1}, the Schanuel conjecture is not needed in the following cases
\begin{itemize}
\item If $H$ has a non trivial exponential part $F_0$.
\item If $H$ has only one irrational power, there is only one $F_i$ as then $p=1$.
\end{itemize}
The problematic case is when $p\geq 2$ and without an exponential part. We will need the notion of decomposability.

\begin{defi}
A rational fraction $G\in\mathbb{K}(x)$ is decomposable when there exists $u\in\overline{\mathbb{K}}(z)$ of degree $\geq 2$ and $F\in\overline{\mathbb{K}}(x)$ such that $G(x)=u(F(x))$. The decomposition of a rational function $G\in\mathbb{K}(x)$ is an indecomposable function $F\in\overline{\mathbb{K}}(x)$ such that
$$\exists u\in\overline{\mathbb{K}}(z), \deg u\geq 2, G(x)=u(F(x)).$$
Moreover $F$ is unique up to homographic transformation.
\end{defi}

\begin{prop}[See \cite{cheze2014decomposition}]
The decomposition of a rational function $G\in\mathbb{K}(x)$ is in $\mathbb{K}(x)$ and is unique up to homographic transformation.
\end{prop}

\begin{proof}[of Theorem \ref{thm1}]
We consider $H$ written under the form
$$H(x)=e^{F_0(x)}A(x)^{1/q}\prod\limits_{i=1}^p F_i(x)^{\lambda_i},$$
where the $\lambda_i$ are traceless independent over $\mathbb{Q}$. Let us consider a $s\in\{0,\dots, p\}$ with non constant $F_s$ and its decomposition of $F_s(x)=u(G(x))$ with $G$ indecomposable and we can choose $u(0)=\infty$. Now we consider the manifold
$$\mathcal{F}_h=\{x\in\mathbb{C}^n, G(x)=h\} \setminus \{\hbox{zero and poles of} F_i,\; i=0\dots p\}.$$
As $G$ is indecomposable, we know that its spectrum, i.e. the set $(\lambda:\mu)\in\mathbb{P}^1$ such that $\lambda \hbox{num}(G)-\mu \hbox{den}(G)$ factorizes, is finite (see \cite{buse2011total}). Thus there exists a disc $\mathcal{D}$ centred in $0$ without $0$, such that $\mathcal{F}_h$ is smooth $\forall h\in\mathcal{D}$.

We now restrict $H\omega$ to $\mathcal{F}_h$ for some $h\in\mathcal{D}$. The integral of this $1$-form of the algebraic manifold $\mathcal{F}_h$ defines a multivalued function. The multivaluation comes from two reasons
\begin{itemize}
\item ``multiplicative multivaluation'': $H\omega$ is hyperexponential, thus a small loop around the zero or pole of a $F_j$ of multiplicity $k$ multiplies the integral by $e^{2i\pi k\lambda_j}$.
\item ``Additive multivaluation'': the manifold $\mathcal{F}_h$ has not a priori a trivial homotopy group, and thus the integral $\int_\gamma H\omega$ for $\gamma$ a closed loop on $\mathcal{F}_h$ could be non zero. Thus such loops add to the integral some constants.
\end{itemize}
Let us first remark that given a closed loop $\gamma(h)$ on $\mathcal{F}_h$, continuous on $h$, the quantity $\int_{\gamma(h)} H\omega$ does not depend on $h$: indeed, the form $H.\omega$ is not only closed on $\mathcal{F}_h$, but in fact on $\mathbb{C}^n$, so the value is constant with respect of continuous deformations of $\gamma$.

The form $H\omega$ is hyperexponential, and thus is defined on a Riemann surface $\mathcal{R}_h$ above $\mathcal{F}_h$. Any closed loop $\gamma$ on $\mathcal{R}_h$ can be decomposed (non commutatively) as closed loops on $\mathcal{F}_h$ and turns around the roots and poles of the $F_i$. Two closed loops $\gamma_1,\gamma_2$ with the same projection on $\mathcal{F}_h$ have the property
$$\int_{\gamma_1} H\omega = \alpha \int_{\gamma_2} H\omega,\qquad \alpha=\exp {2i\pi\left(\frac{n_0}{q}+\sum_{l\neq s} n_l\lambda_l\right)}$$
with $n_i\in\mathbb{Z}$. Indeed, between two points on $\mathcal{R}_h$ with the same projection on $\mathcal{F}_h$, the $1$-form $H\omega$ is simply multiplied by such constant.

Let us note $d=\hbox{dim}(H_1(\mathcal{F}_h,\mathbb{Z}))$. There exist $d$ constants, $u_1,\dots,u_d\in\mathbb{C}$, which are the values of the integral $\int H.\omega$ on a basis of the homotopy group $H_1(\mathcal{F}_h,\mathbb{Z})$. Now the integral $\int_\gamma H.\omega$ on a closed loop $\gamma$ on $\mathcal{R}_h$ will be an integer combination of those $u_j$ and their multiples by elements of the form $\alpha$ , i.e. in the set
$$\mathcal{C}=\left\lbrace\sum_{\underset{j=1\dots d,}{n_i\in\mathbb{Z}}}^{\hbox{finite}} m_{j,n} u_j \exp{2i\pi\left(\frac{n_0}{q}+\sum_{l\neq s} n_l\lambda_l\right)},\;\; m_{j,n}\in\mathbb{Z}\right\rbrace.$$

Let us prove that $\mathcal{C}=\{0\}$. Two cases appear.\\

We can choose $s=0$, i.e. $F_0$ is not constant.\\
We now consider a limit path towards $h=0$. At the limit $h=0$, $\mathcal{F}_h$ is possibly no longer smooth, however the values of the monodromy $\mathcal{C}$ are conserved by passing at the limit because it is constant with respect to $h$. Using the Casorati-Weierstrass theorem, we can choose for the limit of $\exp(u(h))$ any value we want (as $h=0$ is an essential singularity). Thus by choosing two different limit paths, we define a transformation on the space of cycles, and the values of $H\omega$ differ only by a factor $\alpha\in\mathbb{C}^*$. The set $\mathcal{C}$ should thus be stable by multiplication by $\alpha$. As $\alpha$ can be chosen $\alpha$ arbitrary, $\mathcal{C}=\{0\}$.\\

We have $F_0$ constant, and then $p\geq 1$, so we can choose $s\geq 1$ (as else $H$ would be algebraic).\\
We consider a loop around $0$ in $\mathcal{D}$. As $\mathcal{F}_h$ is always smooth, any cycle on $\mathcal{R}_h$ deforms continuously, and thus after coming back to the initial $h$, is again a cycle on $\mathcal{R}_h$. This defines a linear transformation on the homotopy space of $\mathcal{R}_h$. This also multiplies $H.\omega$ by $e^{2i\pi k\lambda_s}$ for some $k\in\mathbb{Z}$. We know that $\int_\gamma H.\omega \in\mathcal{C}$ for any closed loop on $\mathcal{R}_h$, and thus $\mathcal{C}$ should be stable by multiplication by $e^{2i\pi k\lambda_s}$. Thus we have a relation of the form
$$e^{2i\pi k\lambda_s} u=  M u,\quad M\in M_d\left(\mathbb{Z}\left(e^{2i\pi/q},(e^{2i\pi \lambda_j})_{j\neq s}\right)\right).$$
Now if $u\neq 0$, this defines an eigenvector of eigenvalue $e^{2i\pi k\lambda_s}$. Thus, as $e^{2i\pi/q}$ is algebraic, $\chi_M(e^{2i\pi k\lambda_s})=0$ defines a non trivial algebraic relation between the $e^{2i\pi\lambda_j}$. Now according to the Conjecture \ref{conj1}, such relation is not possible. Thus $u=0$, and so $\mathcal{C}=\{0\}$.\\

As $\mathcal{C}=\{0\}$, the integral $\int H.\omega$ has no additive monodromy, and thus is hyperexponential on $\mathcal{F}_h$. So $H^{-1}\int H\omega$ is rational on $\mathcal{F}_h$. We can thus write
$$H^{-1} \int H\omega dx =J(h,x)$$
where $J$ is rational in $x$, and so replacing $h=F(x)$, we obtain relation \eqref{eq4}.

Let us now prove point $2$ of the Theorem. Let us now consider $D_1,\dots,D_{n-1}$, $n-1$ independent rational derivations tangential to $\mathcal{F}_h$ (and thus $D_i(F(x))=0$). We have
$$D_i(H(x))J(F(x),x)+H(x)D_i(J)(F(x),x)=H(x)\sum_{j=1}^n D_{i,j} \omega_j$$
for $i=1\dots n-1$ and then by restriction to $\mathcal{F}_h$
\begin{equation}\label{eq6}
\frac{D_i(H(x))}{H(x)}J(h,x)+D_i(J)(h,x)=\sum_{j=1}^n D_{i,j} \omega_j.
\end{equation}
This is a partial differential system on $\mathcal{F}_h$, and its solution space is an affine vector space of dimension $1$. Let us look at the subspace of rational solutions. The homogeneous solution of the equation is $C(h)H(x)^{-1}$. 

Let us first assume the restriction of $H$ to $\mathcal{F}_h$ is not rational, then equation \eqref{eq6} has at most one rational solution (the homogeneous equation having no rational solution), and as it has at least one due to the previous proof of point $1$ of the Theorem \ref{thm1}, it has exactly one. Now this system can be seen as an integrable connection with coefficients in $\mathbb{K}(h)$, and thus its solution has coefficients in the same field \cite{barkatou2012computing}. Thus $J(h,x)\in\mathbb{K}(h,x)$, and as $\int H\omega=H(x)J(F(x),x)$, the form $H\omega$ is exact.

Let us now assume $H$ is rational on $\mathcal{F}_h$. The space of rational solutions are of the form
$$C(h)H(x)^{-1}+R(h,x),$$
with $R$ a rational function in all its variables, and thus
$$\int H\omega dx=C(F(x))+H(x)R(F(x),x)$$
for some unknown function $C$. By differentiating both sides, we deduce that $C'(F(x))$ should be hyperexponential and that $C'(F(x))/H(x)=T(x)\in\overline{\mathbb{K}}(x)$. As $F$ is indecomposable, the function $C'(z)$ is then hyperexponential. Thus there exists $g\in\overline{\mathbb{K}}(z)$ such that
$$C(z)=\int e^{\int g(z) dz} dz,\quad H(x)=T(x)C'(F(x))$$
\end{proof}

\section{The algorithms}

In this section, we will present explicit algorithms to compute the decomposition of Theorem \ref{thm1} and cohomology basis of Theorem \ref{thm2}.

\subsection{Hyperexponential Decomposition}

\begin{prop}
Consider a transcendental hyperexponential function $H$ with $dH/H\in\mathbb{K}(x)_n$. We write
\begin{equation}\label{eq8}
H(x)=e^{F_0(x)}A(x)^{1/q}F_1(x)^{\lambda_1}\dots F_p(x)^{\lambda_p}
\end{equation}
where $q\in\mathbb{N}^*$, $\lambda_i\in\overline{\mathbb{K}}$ traceless independent on $\mathbb{Q}$, $A,F_i\in \overline{\mathbb{K}}(x)$. If there exists rational functions $F,T\in\overline{\mathbb{K}}(x)$ and $g\in\overline{\mathbb{K}}(z)$ such that
$$H(x)=L(F(x),x),\;\; L(z,x)= T(x)\exp{\int g(z) dz}$$
then $F$ can be chosen to be indecomposable, all the $F_i$ with $i\neq 1$ are rational functions of $F$ and $F,T,g$ have coefficients in $\mathbb{K}$. The algorithm \underline{\sf HyperexponentialDecomposition} computes such a $F,T,g$ if one exists.
\end{prop}

The function $F(x)$ is obtained by decomposition of the rational functions $F_i$. The rational function $T$ is obtained by considering, as in proof of Theorem \ref{thm1}, derivations tangential to $F(x)=h$ and then building a PDE system for $T$. To do it explicitly, we introduce explicit expressions for the $D_i$. A simple way to build them is to consider $x_n$ as an algebraic function in $x_1,\dots,x_{n-1},h$ on the level $F=h$ and consider the derivations in $x_i,i=1\dots n-1$. We then find
$$D_i=\partial_i-\frac{1}{\partial_n Z}\left(\sum\limits_{j=1}^n \partial_j Z)\right) \partial_n$$
where $Z=\hbox{num}(F)-h\hbox{den}(F)\in\mathbb{K}[x,h]$. The polynomial $Z$ should not be constant with respect to $x_n$, but if it is we can always choose another variable: the polynomial $Z$ cannot be constant.\\

\noindent\underline{\sf HyperexponentialDecomposition}\\
\textsf{Input:} A closed $1$-form $dH/H\in\mathbb{K}(x)_n$ with $H$ not algebraic.\\
\textsf{Output:} If possible, rational functions $F,T$ and $g$ such that
$$H(x)=T(x)\exp{\int^{F(x)} g(z) dz}.$$
\begin{enumerate}
\item Apply algorithm \underline{\sf RationalIntegration} and obtain representation \eqref{eq8} with constant field $\mathbb{L}$.
\item If $F_0$ non constant, take for $F$ the decomposition of $F_0$. Else consider a non constant $F_i$, and find $m\in\{1,\dots,[\mathbb{L}:\mathbb{K}]\}$ such that
\begin{equation}\label{eq10}
\sum_{\sigma\in Gal(\mathbb{L}:\mathbb{K})} \sigma(F_i)^m \in\mathbb{K}(x)
\end{equation}
is not constant, and take for $F$ its decomposition.
\item Solve in $\mathbb{K}(h)(x_1,\dots,x_{n-1})$ the system
$$D_i(T)=\frac{D_i(H)}{H} T,\;\; i=1\dots n-1$$
and note $\tilde{T}\in \mathbb{K}(h)(x_1,\dots,x_{n-1})$ a solution, and $T(x)=\tilde{T}(F(x),x)$. If none exists, return ``None''.
\item Find $g\in\mathbb{K}(z)$ such that
$$\frac{d(H(x)/T(x))}{H(x)/T(x)}=g(F(x)) dF(x).$$
If such $g$ does not exist, return ``None'' else return $F,T,g$.
\end{enumerate}

\begin{proof}

We can compute the representation \eqref{eq8} in step $1$. If $H(x)=T(x)\exp{\int^{F(x)} g(z) dz}$ for some $F\in\overline{\mathbb{K}}(x)$, then its restriction on a level set $\mathcal{F}_h$ of $F$ should be rational. This implies that all the $F_i$ restricted to $\mathcal{F}_h$ should be constant. Thus all such $F_i$ are algebraic functions of a same $F$. The composition of a hyperexponential function with a rational function is hyperexponential. Thus $F$ can always be chosen indecomposable, up to changing $g$. Then according to \cite{cheze2014decomposition}, the $F_i$ are rational functions of $F$.

Let us go to step $2$. If $F_0\neq 0$, its decomposition $\tilde{F}_0\in\mathbb{K}(x)$ suits us, and thus we take $F=\tilde{F}_0$. Else consider the action of $\sigma \in Gal(\mathbb{L}:\mathbb{K})$ on the $\lambda_i,F_i$. As the traceless property is stable by Galois action, a conjugate of representation \eqref{eq8} is also a suitable representation. As $A,F_0\in\mathbb{K}(x)$, we have
$$\sum \lambda_i\ln F_i=\sum \sigma(\lambda_i)\ln \sigma(F_i).$$
The left-hand side is function of $F$, the right-hand side function of $\sigma(F)$, so by uniqueness, they are equal up to homography. So all the $\sigma(F_i)$ are function of $F$. Now considering the sums \eqref{eq10} for some non constant $F_i$, they are function of $F$. At least one of these sums is non constant as else $F_i$ would be constant. Thus $F$ can be obtained as the decomposition of such a sum, and so $F\in\mathbb{K}(x)$.

We now continue the algorithm to find $T,g$ and to prove that their coefficients are in $\mathbb{K}$. Restricted to the level $F(x)=h$, we want $H(x)=T(x) \exp{\int g(h) dh}$. As the derivations $D_i$ are tangential to $\mathcal{F}_h$, we obtain taking the logarithmic derivative both sides
$$\frac{D_i(T)}{T}=\frac{D_i(H)}{H}$$
We can see $T$ as an element of $\overline{\mathbb{K}(h)}(x_1,\dots,x_{n-1})[x_n]/(Z)$, as well as the right-hand side of the equation. Noting
$$T(h,x)=\sum\limits_{i=0}^{d-1} T_i(x_1,\dots,x_{n-1},h) x_n^i$$
with $d$ the degree in $x_n$ of $Z$, the system becomes an integrable connection on $T_0,\dots,T_{d-1}$ in $n-1$ variables $x_1,\dots,x_{n-1}$ over the base coefficient field $\mathbb{K}(h)$. Such system can be solved, and without extending the base coefficient field $\mathbb{K}(h)$ \cite{barkatou2012computing}. Thus if a rational solution in $x$ exists, then it can be chosen also rational in $h$. This gives $T$ in step $4$.

Now the $T$ obtained in step $3$ is valid up to multiplication by an arbitrary function of $F(x)$. However, as $T$ is rational and $F$ indecomposable, it is valid up to multiplication by a rational function of $F(x)$. Such factor can be taken care by the hyperexponential part, thus if a decomposition of $H$ exist, a decomposition with this $T$ will exist too. Step $4$ comes from the logarithmic differential of the relation $H(x)/T(x)$, giving
$$\frac{d(H(x)/T(x))}{H(x)/T(x)}= \frac{d(L(F(x),x)/T(x))}{L(F(x),x)/T(x)}= g(F(x)) dF(x)$$
which defines uniquely $g$.
\end{proof}

\subsection{Liouvillian Decomposition}

The function $F$ of Theorem \ref{thm1} will be obtained using the algorithm \underline{\sf HyperexponentialDecomposition}. If it does not return a solution, then $H\omega$ is exact and thus can be integrated with hyperexponential functions. Else it returns $(F,T,g)$, and equation \eqref{eq6} can be simplified in
\begin{equation}\label{eq7}
D_i(T(x)J(h,x))=T(x)\sum_{j=1}^n D_{i,j} \omega_j,\;\;i=1\dots n-1.
\end{equation}
Now the function $T(x)J(h,x)$ can be recovered up to the addition of an arbitrary function of $h$ by an indefinite integration on the level set $\mathcal{F}_h$. As the function $R(x)T(x)$ in equation \eqref{eq4} equals to $T(x)J(h,x)$ up to the addition of an arbitrary function of $h$, we then obtain candidates for $R$.\\

\noindent\underline{\sf LiouvillianDecomposition}\\
\textsf{Input:} A closed $1$-form $dH/H\in\mathbb{K}(x)_n$, and a $1$-form $\omega\in\mathbb{K}(x)_n$ such that $H\omega$ is closed.\\
\textsf{Output:} Rational functions $F,R,f,g$ satisfying \eqref{eq5b}.
\begin{enumerate}
\item Solve the system in $\mathbb{K}(x)$
$$\partial_i R=\omega_i-\frac{\partial_i H}{H} R,\quad i=1\dots n.$$
If it has a solution $R$, return $[0,R,0,0]$.
\item Apply \underline{\sf HyperexponentialDecomposition} to $H$, get $F,T,g$.
\item Solve the system \eqref{eq7} for $J(h,x)\in \mathbb{K}(h)(x_1,\dots,x_{n-1})$. Take one solution, and note it $\tilde{R}$ and
$$R(x)=\tilde{R}(F(x),x_1,\dots,x_{n-1}).$$
\item Find $f\in\mathbb{K}(z)$ such that
$$f(F)dF= T\omega -TdR-TR\frac{dH}{H}$$
\item Return $[F,R,f,g]$.
\end{enumerate}

\begin{prop}
The algorithm \underline{\sf LiouvillianDecomposition} takes in input $dH/H,\omega\in\mathbb{K}(x)_n$ and computes $F,R,f,g$ with coefficients in $\mathbb{K}$ satisfying equation \eqref{eq5b}.
\end{prop}

\begin{proof}
Step $1$ tests if the form is exact. If it is, it returns $f=0,g=0,F=0$ and $R$. This satisfies equation \eqref{eq5b}. In step $2$, the algorithm \underline{\sf HyperexponentialDecomposition} cannot return ``None'' as then $H\omega$ would be exact.

In step $3$, we solve the system. As $H\omega$ is not exact, then the system has a rational solution $\tilde{R}$, and so $T(x)\tilde{R}(h,x)$ gives an expression of $T(x)J(h,x)$ up to the addition of an arbitrary function of $h$. Thus we have
$$ H(x)^{-1}\int H\omega = J(F(x),x)=R(x)+\frac{\nu(F(x))}{T(x)}.$$
As we already know that $\nu$ can be written under the form
$$\nu(z)=e^{-\int g(z) dz} \int f(z)e^{ \int g(z) dz} dz$$
thanks to point $2$ of Theorem \ref{thm1}, this implies that
$$\int H\omega=R(x)H(x)+ \int^{F(x)}\!\! f(z)e^{ \int g(z) dz} dz.$$
We differentiate the relation, giving
$$H\omega= HdR+RH\frac{dH}{H}+ f(F)e^{ \int^F g(z) dz}  dF$$
$$T\omega =TdR+TR\frac{dH}{H}+f(F)dF$$
and thus $f(F)dF= T\omega -TdR-TR\frac{dH}{H}$.
This is the relation solved in step $4$.
\end{proof}

\subsection{Cohomology algorithm}

We now consider a square free polynomial $S\in\mathbb{K}[x]$, a transcendental hyperexponential function $H$, and $D$ denominator of $dH/H$. We know that the pullback function $F$ depends only on $H$, so for any closed $1$-form with coefficients in $H\mathbb{K}[x,1/(SD)]$, they will have the same $F$. Moreover, the hyperexponential part of the decomposition of Theorem \ref{thm1} corresponds to the integral of an exact form, and thus is irrelevant for the computation of the cohomology. The only part left to control is $f$. Thus it is necessary to understand the cohomology of one variable hyperexponential $1$-forms, which is done in \cite{bostan2013hermite}.

\begin{defi}
Consider $g\in\mathbb{K}(z)$ a rational function. The kernel $K\in \mathbb{K}(z)$ of $g$ is equal to $g$ minus all poles of order $1$ with integer residues. The function $g$ is said to be differentially reduced if $K=g$. The shell of $g$ is $\exp{ \int g-K dz}\in\mathbb{K}(z)$.
\end{defi}

\begin{prop}\label{prop3}
Consider $g=g_1/g_2\in\mathbb{K}(z)$ differentially reduced with $d_1=\deg g_1,\; d_2=\deg g_2,\; d_1-d_2 \leq -2$, and $Q \in \mathbb{K}[z]$ square free coprime with $g_2$. A basis of the cohomology of $\exp{ \int g(z) dz} \mathbb{K}[z,1/(Qg_2)]dz$ is given by
$$(z^i/Q)_{i=0\dots \deg Q-1},(z^i/g_2)_{i=0\dots d_2-1,\;i\neq d_1}$$
\end{prop}
\begin{proof}
We first use Lemmas $6,16$ of \cite{bostan2013hermite}, reducing the rational part of the form to $q/Q+p/g_2$ with $\deg q<\deg Q$. Now using Lemma $8$ of \cite{bostan2013hermite}, we can now reduce $p$ to a vector space $\mathcal{N}_K$. When $d_1-d_2 \leq -2$ (case $3$), the basis of $\mathcal{N}_K$ is $z^i,\; i=0\dots d_2-1,\;i\neq d_1$, giving Proposition \ref{prop3}.
\end{proof}

The condition $d_1-d_2\leq -2$ is equivalent to ask that the form $g(z)dz$ has not a pole at infinity. The function $g$ used will be the output of \underline{\sf HyperexponentialDecomposition}, and as $F$ is unique only up to homographic transformation, it is always possible to make a homographic transformation of $g(z)dz$ and $F$ to satisfy this condition. Now the function $g$ can always be assumed differentially reduced by multiplying $T$ by a suitable rational function of $F$.  Then applying \underline{\sf LiouvillianDecomposition}, we obtain the ``non integrable part'' $\int f(z)\exp{\int g(z) dz}dz$. The poles of $f$ will lead after substitution by $F$ to poles of $\omega$, and as their location is controlled by $SD$, this will allow us to control the poles of $f$.\\

\noindent\underline{\sf CohomologyBasis}\\
\textsf{Input:} A closed form $dH/H\in\mathbb{K}(x)_n$ with $H$ transcendental, $S\in\mathbb{K}[x]$ square free coprime with $D$ denominator of $dH/H$.\\
\textsf{Output:} A list of forms $(\omega_i)_{i=1\dots r}$ such that $(H\omega_i)_{i=1\dots r}$ is a basis of the cohomology of $H\mathbb{K}[x,1/(SD)]_n$.\\
\begin{enumerate}
\item Apply \underline{\sf HyperexponentialDecomposition} to $dH/H$. If it returns ``None'', return $[\;\!]$. Else obtain $F,T,g$.
\item Find $h$ homography such that $h'(z)g(h(z))$ is of degree $\leq -2$ and $\hbox{den}(h^{-1}(F))$ does not divide a power of $SD$. Replace $g$ by $h'(z)g(h(z))$ and $F$ by $h^{-1}(F)$.
\item Compute the kernel $k$ and shell $s$ of $g$, and replace $g$ by $k$, and $T$ by $Ts(F)$.
\item For all simple roots $\alpha$ of $\hbox{den}(g)$ with rational residue, replace $g$ by $g+m/(z-\alpha)$ and $T$ by $T/(F-\alpha)^m$ with $m\in \mathbb{N}$ such that the residue of $g$ at $\alpha$ is positive.
\item Note $m$ the residue at infinity of $g$, take a root $\alpha$ of $\hbox{den}(g)$ of order $\geq 2$ or with irrational residue, and replace $g$ by $g-m/(z-\alpha)$ and $T$ by $T(F-\alpha)^m$.
\item Noting $g=g_1/g_2$, $d_1=\deg g_1,\; d_2=\deg g_2$, compute
$$\Sigma=\{c\in\bar{\mathbb{K}},\exists u,\; \hbox{num}(F-c) \mid (SD)^u\},\; Q(z)=\!\!\!\!\!\!\!\! \prod_{c\in\Sigma\setminus g_2^{-1}(0)} \!\!\!\!\!\!\!\! z-c$$
\item Return
$$\sum_{\sigma\in\hbox{Gal}(\mathbb{K}(\alpha):\mathbb{K})} \sigma\left(\frac{F^idF}{TQ(F)}\right),\quad i=0\dots \deg Q-2, $$
$$\sum_{\sigma\in\hbox{Gal}(\mathbb{K}(\alpha):\mathbb{K})} \sigma\left(\frac{F^idF}{Tg_2(F)}\right),\quad i=0\dots d_2-2, i\neq d_1,$$
$$\sum_{\sigma\in\hbox{Gal}(\mathbb{K}(\alpha):\mathbb{K})} \sigma\left(\frac{F^{\deg Q-1}dF}{TQ(F)}-\hbox{lc}(g_2)\frac{F^{d_2-1}dF}{Tg_2(F)}\right)$$
\end{enumerate}

\begin{proof}
In step $1$, if \underline{\sf HyperexponentialDecomposition} re\-tur\-ns ``None'' when applied to $H$, this implies that any closed $1$-form $H\omega,\; \omega\in\mathbb{K}(x)_n$ is exact. Thus the cohomology of $H\mathbb{K}[x,1/(SD)]_n$ in particular is trivial, and thus the returned basis is empty.
In step $2$, we want a homographic transformation to $g(z)dz$ so that it is not singular at $\infty$, which is equivalent to $h'(z)g(h(z))$ being of degree $\leq -2$. This constrains only $h$ of not sending a singular point of $g(z)dz$ to infinity. The denominators $\hbox{den}(h^{-1}(F))$ are linear combinations of numerator and denominators of $F$, and thus only finitely many such combinations can divide a power of $SD$. Thus it is always possible to find a suitable $h$. Now replacing $F$ by $h^{-1}(F)$, we ensure that $F,T,g$ is a decomposition of $H$. 
In step $3$, we differentially reduce $g$ and change $T$ accordingly so that $F,T,g$ is now a decomposition of $H$ with $g$ differentially reduced.  The differential reduction may have changed the residue at infinity to an integer $m$. In step $4$, we shift rational residues by integers and modify accordingly $T$, and this does not create integer residues, so $g$ is still differentially reduced. In step $5$, we modify the residue of a pole of $g$ (keeping it differentially reduced) such that the residue at infinity of $g$ is zero (so $\deg g\leq -2$), and change $T$ accordingly so that $F,T,g$ is still a decomposition of $H$. As this pole is of order $\geq 2$ or with irrational residue, rational residues of $g$ are still positive.

We know that any closed form $H\omega$ in $H\mathbb{K}[x,1/(SD)]_n$ can be reduced modulo exact forms to $f(F)\exp {\int^F g(z) dz} dF$ by Theorem \ref{thm1}. We can now apply Proposition \ref{prop3}. The denominators of $f$ can be reduced to $g_2$ and simple poles $\notin g_2^{-1}(0)$. Now considering such a pole $c$ of $f$, we know that $\exp \int g(z) dz$ is smooth at $z=c$ and is of order $1$ to $f$. Thus thus integral $\int f(F)\exp {\int^F g(z) dz} dF$ has a logarithmic singularity along the curve $F=c$. This property is conserved by adding terms in $H\mathbb{K}(x)_n$, and thus so is $\int H\omega$. This implies that $H\omega$ has a pole along $F=c$, and thus $\hbox{num}(F-c)$ divides a power of $SD$. All such possible $c$ are computed in step $6$, and as the $\hbox{num}(F-c)$ are coprime for different $c$, the set $\Sigma$ is finite. Thus we can define the polynomial $Q\in\mathbb{K}[z]$ as the sets $\Sigma,g_2^{-1}(0)$ are invariant by Galois action.

We can now apply Proposition \ref{prop3} with $g,Q$ obtained in steps $5,6$. We obtain a vector space of possible $f$. However, some elements of this vector space have a non zero residue at infinity. A $H\omega\in H\mathbb{K}[x,1/(SD)]_n$ cannot have a singularity along $\hbox{den}(F)=0$ as $\hbox{den}(F)$ does not divide a power of $SD$. Thus its integral cannot have a logarithmic singularity along $\hbox{den}(F)=0$, and so $f(z)\exp \int g(z) dz$ cannot have a logarithmic singularity at $z=\infty$. This ensures that the possible $f$ will not have residues at infinity, and so removes after basis change a single basis element of the cohomology. As all other elements have degree $\leq -2$, the $F$ substitution and differentiation does not make appear $\hbox{den}(F)$ in the denominators. 

Consider $\mathcal{C}$ a curve on which $T=0$. On $\mathcal{C}$, the function $\exp{\int^F g(z) dz}$ can be smooth, essential singular, irrational ramified or rational ramified with positive exponent. In all cases, $H=T\exp{\int^F g(z) dz}$ still vanishes or is singular on $\mathcal{C}$, and thus $dH/H$ is singular on $\mathcal{C}$. Thus $\hbox{num}(T)$ divides $D$, and so
$$\frac{F^idF}{TQ(F)},\frac{F^idF}{Tg_2(F)},\frac{F^{\deg Q-1}dF}{TQ(F)}-\hbox{lc}(g_2)\frac{F^{d_2-1}dF}{Tg_2(F)}$$
of step $7$ are in $\mathbb{K}(\alpha)[x,1/(SD)]_n$.

If for any $\omega\in \mathbb{K}(\alpha)[x,1/(SD)]_n$ with $H\omega$ closed, we can write $\omega=\omega_e+\omega_c$ with $H\omega_e$ exact and $\omega_c$ in a vector space $\hbox{Span}(v_1,\dots,v_r) \subset \mathbb{K}(\alpha)[x,1/(SD)]_n$, taking the sum over the Galois conjugates of $\alpha$ and dividing by the degree of the extension gives us a similar decomposition with $\omega_c$ in $\hbox{Span}(\sum_{\hbox{Gal}} v_1,\dots,\sum_{\hbox{Gal}} v_r) \subset \mathbb{K}[x,1/(SD)]_n$. Thus step $7$ returns a basis of the cohomology of $H\mathbb{K}[x,1/(SD)]_n$.
\end{proof}

\section{Applications and Examples}

\begin{proof}[of Corollary \ref{cor1}]
If system \eqref{eq3} admits a Liouvillian first integral, it can be written
$$I(x,y)=\int \mathcal{R} dx-\mathcal{R} Vdy$$
where $\mathcal{R}$ is a hyperexponential function, called the integrating factor. If $\mathcal{R}$ is algebraic, then $I$ is $k$-Darbouxian which is forbidden by hypothesis. If $\mathcal{R} dx-\mathcal{R} Vdy$ is exact, its integral is hyperexponential, and thus $I$ is the exponential of a Darbouxian first integral, again forbidden by hypothesis. Thus we can apply Theorem \ref{thm1} and write
$$I(x,y)= \int^{F(x,y)} \!\!\!\!\!\!\!\!\! f(z) e^{\int g(z) dz} dz+T(x,y)R(x,y)e^{\int^{F(x,y)} \!\!\!\! g(z) dz}$$
Let us note $X=F(x,y),\; Y=R(x,y)T(x,y)$. In these new variables, the first integral writes
$$I(X,Y)=\int f(X)e^{\int g(X) dX} dX+e^{\int g(X) dX} Y$$
Differentiating $I$ in $X,Y$, we obtain $-\frac{\partial Y}{\partial X}=f(X)+g(X)Y.\!\!\!$
\end{proof}

The algorithms presented in this article are implemented for $\mathbb{K}=\mathbb{Q}$ (but without restrictions on $\mathbb{L}$) in Maple and are available on \url{http://combot.perso.math.cnrs.fr/}. All the timings are around $1$s. The dominant cost is the possible large degree of $F$ and field extension $\mathbb{L}$.\\

Exemple $1$: $dH/H=$\\
$$\frac{2x_1^3-12x_1^2x_2-3x_1^2+6x_2^2}{3x_1^2(x_1^2-2x_2^2)} dx_1+ \frac{4(3x_1-x_2)}{3(x_1^2-2x_2^2)}dx_2+\frac{1}{x_3}dx_3$$
\underline{\sf RationalIntegration} returns
$$\frac{1}{x_1}, \left[ \frac{1}{3}, x_3^3(x_1^2-2x_2^2)\right], \left[ \lambda_{2,2}, \frac{\lambda_{2,2}x_1+2x_2}{-\lambda_{2,2}x_1+2x_2} \right],$$
$$\mathbb{L}\simeq\mathbb{Q}[\lambda_{2,1},\lambda_{2,2}]/<-\lambda_{2,1}-\lambda_{2,2}, \lambda_{2,1}\lambda_{2,2}+2 >.$$

Exemple $2$:\\
$$\frac{dH}{H}=2\frac{7x_1-2}{x_1^2-2}dx_1-\frac{4a}{x_2^2-2}dx_2$$
For $a=4$, \underline{\sf HyperexponentialDecomposition} returns
$$H=(x_1^2-2)^7e^{\int^F -\frac{4}{2z^2-2} dz},$$
$$F=\frac{4x_1x_2^3+x_2^4+8x_1x_2+12x_2^2+4}{x_1x_2^4+12x_1x_2^2+8x_2^3+4x_1+16x_2}$$
Remark that for $a\in\mathbb{N}^*$, the degree of $F$ is $a$. This is due to the fact that the degree of $F$ is related to the height of integer relation between the residue, here depending on $a$.\\

Exemple $3$:\\
The differential system
$$\dot{x}_1=-3x_1^6x_2^2-9x_1^4x_2^4-9x_1^2x_2^6-3x_2^8+2x_1^6-2x_1^5x_2-$$
$$6x_1^4x_2^2-2x_1^2x_2^4-6x_1x_2^5-2x_2^6-2x_1^4+4x_1^3x_2+4x_1x_2^3+2x_2^4,$$
$$\dot{x}_2=3x_1^8+9x_1^6x_2^2+9x_1^4x_2^4+3x_1^2x_2^6+2x_1^6+6x_1^5x_2+$$
$$2x_1^4x_2^2+6x_1^2x_2^4+2x_1x_2^5-2x_2^6-2x_1^4-4x_1^3x_2-4x_1x_2^3+2x_2^4$$
admits a Liouvillian first integral $J$ with integrating factor $H$ given by $-dH/(2H)=$
$$\frac{3x_1^5+6x_1^3x_2^2+3x_1x_2^4-x_1^3-3x_1^2x_2-x_1x_2^2+x_2^3}{(x_1^2+x_2^2)^3}dx_1$$
$$+\frac{3x_1^4x_2+6x_1^2x_2^3+3x_2^5+x_1^3-x_1^2x_2-3x_1x_2^2-x_2^3}{(x_1^2+x_2^2)^3}dx_2$$
The algorithm \underline{\sf LiouvillianDecomposition} returns
$$H=-\frac{8}{(x_1+x_2)^3}e^{\int^F -\frac{3z^2-2}{z^3} dz},\;\; F=-\frac{x_1^2+x_2^2}{x_1+x_2}$$
$$J=\int^F -4(33z^4+22z^2-4)ze^{\int -\frac{3z^2-2}{z^3} dz} dz-$$
$$H\frac{9x_1^6-6x_1^5x_2+27x_1^4x_2^2-4x_1^3x_2^3+27x_1^2x_2^4-6x_1x_2^5+9x_2^6}{2(x_1^2+x_2^2)^{-3}(x_1+x_2)^3}$$
The differential system is thus rationally equivalent to
$$\frac{\partial Y}{\partial X}= 4(33X^4+22X^2-4) +\frac{3X^2-2}{X^3}Y$$

Exemple $4$:\\
We consider $H$ given by $dH/(2H)=(x_1+x_2)\times $
$$\left(\frac{x_1^2+2x_1x_2-x_2^2}{(x_1^2+x_2^2)^3}dx_1-\frac{x_1^2-2x_1x_2-x_2^2}{(x_1^2+x_2^2)^3}dx_2 \right)$$
and
$$S=(x_1^2+x_2^2+x_1+x_2)(x_1^2+x_2^2-x_1-x_2)(x_1+2x_2).$$
The algorithm \underline{\sf CohomologyBasis} returns
$$\frac{120(x_1^2+2x_1x_2-x_2^2)}{(x_1^2+x_2^2)^2-(x_1+x_2)^2}dx_1 -\frac{120(x_1^2-2x_1x_2-x_2^2)}{(x_1^2+x_2^2)^2-(x_1+x_2)^2}dx_2,$$
$$\frac{(x_1^2+2x_1x_2-x_2^2)(11x_1^2+11x_2^2+x_1+x_2)}{(x_1^2+x_2^2)^3}dx_1-$$
$$\frac{(11x_1^2+11x_2^2+x_1+x_2)(x_1^2-2x_1x_2-x_2^2)}{(x_1^2+x_2^2)^3},$$
$$\frac{(11x_1^2+11x_2^2+x_1+x_2)(x_1^2+2x_1x_2-x_2^2)}{(x_1^2+x_2^2+x_1+x_2)(x_1^2+x_2^2-x_1-x_2)(x_1^2+x_2^2)}dx_1-$$
$$\frac{(11x_1^2+11x_2^2+x_1+x_2)(x_1^2-2x_1x_2-x_2^2)}{(x_1^2+x_2^2+x_1+x_2)(x_1^2+x_2^2-x_1-x_2)(x_1^2+x_2^2)}dx_2.$$
As we see, $x_1+2x_2$ never appears in the poles of these forms. This is because such pole cannot have a non zero residue, and thus is always reducible modulo exact forms.


\bibliographystyle{abbrv}

\begin{small}
\bibliography{hyperexponential}
\end{small}

\end{document}